\documentclass[reqno]{amsart}
\usepackage{amscd}
\usepackage{amssymb}
\usepackage{amsmath}
\usepackage{amsthm}
\usepackage{latexsym}
\usepackage{upref}
\usepackage{epsfig}
\usepackage{mathrsfs}
\usepackage{float}
\usepackage[colorlinks,urlcolor=black]{hyperref}
\usepackage{graphics,graphicx,color}
\usepackage{floatflt}
\usepackage{indentfirst}
\usepackage[cmtip,arrow]{xy}
\usepackage{pb-diagram,pb-xy}
\usepackage{cite}

\usepackage{enumitem}
\setlist{font=\normalfont}
\newtheorem{theorem}{Theorem}[section]
\newtheorem{proposition}[theorem]{Proposition}
\newtheorem{corollary}[theorem]{Corollary}
\newtheorem{lemma}[theorem]{Lemma}
\theoremstyle{definition}

\newtheorem{definition}[theorem]{Definition}

\numberwithin{equation}{section}

\newcommand{\R}{\mathbb{R}}
\newcommand{\N}{\mathbb{N}}
\newcommand{\Z}{\mathbb{Z}}

\newcommand{\Gam}{\Gamma}

\newcommand{\vphi}{\varphi}

\newcommand{\aut}[1]{\mathrm{Aut}\,{(#1)}}
\newcommand{\gyr}[2]{{\mathrm{gyr}[{#1}]}{#2}}

\newcommand{\id}[1]{\mathrm{id}_{#1}}
\newcommand{\res}[2]{{#1}\big|_{{#2}}}
\newcommand{\cols}[1]{\mathbf{\mathcal{#1}}}

\newcommand{\lcap}[2]{\displaystyle\bigcap_{#1}^{#2}}

\newcommand{\gen}[1]{\langle#1\rangle}
\newcommand{\set}[1]{\{#1\}}
\newcommand{\cset}[2]{\set{{#1}\colon{#2}}}
\newcommand{\abs}[1]{|#1|}

\begin{document}
\title[Lagrange's Theorem for Gyrogroups and the Cauchy Property]{Lagrange's Theorem for Gyrogroup\\ and the Cauchy Property}

%author one information
\author{Teerapong Suksumran}
\address{Department of Mathematics and Computer Science, Faculty of Science,
Chulalongkorn University, Phyathai Road, Patumwan, Bangkok 10330,
Thailand} \email{sk\_teer@yahoo.com \textrm{and
}kwiboonton@gmail.com}

%author two information
\author{Keng Wiboonton}
\address{}
\email{}

\begin{abstract}
In this paper, we prove a version of Lagrange's theorem for
gyro-groups and use this result to prove that gyrogroups of
particular orders have the strong Cauchy property.
\end{abstract}
 \maketitle

\section{Introduction}
\par Lagrange's theorem (that the order of any subgroup of a finite group $\Gam$ \mbox{divides} the order
of $\Gam$) is well known in group theory and has impact on several
branches of mathematics, especially finite group theory,
combinatorics, and number \mbox{theory}. Lagrange's theorem proves
useful for unraveling mathematical structures. For \mbox{instance},
it is used to prove that any finite field must have prime power
order. Certain classification theorems of finite groups arise as an
application of Lagrange's theorem \cite{JGDM1995, AS1997, JG2001}.
Further, Fermat little's theorem and Euler's theorem may be viewed
as a consequence of this theorem. Also relevant are the
orbit-stabilizer theorem and the Cauchy-Frobenius lemma (or
Burnside's lemma). A history of Lagrange's theorem on groups can be
found in \cite{RR2001}.

\par In loop theory, the Lagrange property becomes a nontrivial issue.
For example, whether Lagrange's theorem holds for Moufang loops was
an open problem in the theory of Moufang loops for more than four
decades \cite[p. 43]{OCMKARPV2003}. This problem was answered in the
affirmative by Grishkov and Zavarnitsine \cite{AGAZ2005}. In fact,
not every loop satisfies the Lagrange property as one can construct
a loop of order five containing a subloop of order two.
Nevertheless, some loops satisfy the Lagrange property.

\par Baumeister and Stein \cite{BBAS2011} proved a version of
Lagrange's theorem for Bruck loops by studying in detail the
structure of a finite Bruck loop. Foguel et al. \cite{TFMKJP2006}
proved that left Bol loops of \textit{odd} order satisfy the strong
Lagrange property. It is, however, still an open problem whether or
not Bol loops satisfy the Lagrange property \cite[p. 592]{TFMK2010}.
In the same spirit, we focus on the Lagrange property for
\textit{gyrogroups} or \textit{left Bol loops with the
$\mathrm{A}_\ell$-property} in the loop literature. In
\cite{TSKW2014}, we proved that the order of an
\textit{L-subgyrogroup} of a finite gyrogroup $G$ divides the order
of $G$. In this paper, we extend this result by proving that the
order of \textit{any} subgyrogroup of $G$ divides the order of $G$,
see Theorem \ref{Thm: Lagrange's Theorem for Gyrogroup}.

\par A gyrogroup is a group-like structure, introduced by Ungar,
arising as an algebraic structure that the set of relativistically
admissible vectors in $\R^3$ with  Einstein addition encodes
\cite{AU2007SD}. The origin of a gyrogroup is described in
\cite[Chapter 1]{AU2009MC}. There are two prime examples of
gyrogroups, namely the \textit{Einstein gyrogroup}, which consists
of the relativistic ball in $\R^3$ with Einstein addition
\cite{AU2007SD}, and the \textit{M\"{o}bius gyrogroup}, which
consists of the complex unit disk with M\"{o}bius addition
\cite{AU2008MAA}.

\par In this paper, we prove that Lagrange's theorem holds for
gyrogroups and apply this result to show that finite gyrogroups of
particular orders have the Cauchy property. Our results are strongly
based on results by Foguel and Ungar \cite{TFAU2000} and Baumeister
and Stein \cite{BBAS2011}. For basic terminology and definitions in
loop theory, we refer the reader to \cite{HK2002, RB1971, HP1991}.

\section{Gyrogroups}
\par In this section, we summarize definitions and basic properties
of gyrogroups. Much of this section can be found in \cite{AU2008}.

\par Let $(G, \oplus)$ be a magma. Denote the group of automorphisms
of $G$ with respect to $\oplus$ by $\aut{G,\oplus}$.
\begin{definition}[{\hskip-0.5pt}\cite{AU2008}]
\label{Def: Main definition of gyrogroup} A magma $(G,\oplus)$ is a
\emph{gyrogroup} if its binary operation satisfies the following
axioms:
\begin{enumerate}
    \item [(G1)] $\exists0\in G\,\forall a\in G$, $0\oplus a =
    a$;\hfill(left identity)
    \item [(G2)] $\forall a\in G\,\exists b\in G$, $b\oplus a = 0$;\hfill(left inverse)
    \item [(G3)] $\forall a, b\in G\,\exists\gyr{a, b}{}\in\aut{G,\oplus}\,\forall c\in G$,
    \begin{equation}\tag{left gyroassociative law} a\oplus (b\oplus c) = (a\oplus b)\oplus\gyr{a,
    b}{c};\end{equation}
    \item [(G4)] $\forall a, b\in G$, $\gyr{a, b}{} = \gyr{a\oplus b,
    b}{}$.\hfill(left loop property)
\end{enumerate}
\end{definition}

\par The following theorem gives a characterization of a gyrogroup.
\begin{theorem}[{\hskip-0.5pt}\cite{TFAU2000}]
\label{Thm: Alternative def. of gyrogroup} Suppose that $(G,\oplus)$
is a magma. Then $(G,\oplus)$ is a gyrogroup if and only if
$(G,\oplus)$ satisfies the following properties:
\begin{enumerate}
\item [(g1)] $\exists 0\in G\forall a\in G, 0\oplus a = a$ and $a\oplus 0 =
a$; \hfill\emph{(two-sided identity)}
\item [(g2)] $\forall a\in G\exists b\in G, b\oplus a = 0$ and $a\oplus b = 0$.\hfill\emph{(two-sided inverse)}\\
 For $a, b, c\in G$, define
\begin{equation}\tag{gyrator identity}
 \gyr{a, b}{c} = \ominus(a\oplus
b)\oplus(a\oplus (b\oplus c)),\end{equation} then
\item [(g3)] $\gyr{a,
b}{}\in\aut{G,\oplus}$;\hfill\emph{(gyroautomorphism)}
\item [(g3a)] $a\oplus (b\oplus c) = (a\oplus b)\oplus\gyr{a, b}{c}$;\hfill\emph{(left gyroassociative
law)}
\item [(g3b)] $(a\oplus b)\oplus c = a\oplus (b\oplus\gyr{b, a}{c})$;\hfill\emph{(right gyroassociative
law)}
\item [(g4a)] $\gyr{a, b}{} = \gyr{a\oplus b, b}{}$;\hfill\emph{(left loop
property)}
\item [(g4b)] $\gyr{a, b}{} = \gyr{a, b\oplus a}{}$.\hfill\emph{(right loop
property)}
\end{enumerate}
\end{theorem}

\begin{definition}[{\hskip-0.5pt}\cite{AU2008}]
A gyrogroup $G$ having the additional property that
\begin{equation}\tag{gyrocommutative law}
a\oplus b = \gyr{a, b}({b\oplus a})
\end{equation} for all $a, b\in G$ is called a \emph{gyrocommutative gyrogroup}.
\end{definition}

\par The \emph{gyrogroup cooperation}, $\boxplus$, is defined by the
equation
\begin{equation}\label{Eqn: Cooperation}
a\boxplus b = a\oplus\gyr{a, \ominus b}{b}, \hskip0.5cm a, b\in G.
\end{equation}
\begin{theorem}[{\hskip-0.5pt}\cite{AU2008}]\label{Thm: Linear equations in gyrogroup}
Let $G$ be a gyrogroup and let $a, b\in G$. The unique solution of
the equation $a\oplus x = b$ in $G$ for the unknown $x$ is $x =
\ominus a\oplus b$, and the unique solution of the equation $x\oplus
a = b$ in $G$ for the unknown $x$ is $x = b\boxplus (\ominus a)$.
\end{theorem}

\par By Theorem \ref{Thm: Linear equations in gyrogroup}, the following cancellation laws hold in
gyrogroups.
\begin{theorem}[{\hskip-0.5pt}\cite{AU2008}]
Let $G$ be a gyrogroup. For all $a, b, c\in G$,
\begin{enumerate}
    \item $a\oplus b = a\oplus c$ implies $b = c$;\hfill\emph{(general left cancellation
    law)}
    \item $\ominus a\oplus(a\oplus b) = b$;\hfill\emph{(left cancellation law)}
    \item $(b\ominus a)\boxplus a = b$;\hfill\emph{(right cancellation law I)}
    \item $(b\boxplus (\ominus a))\oplus a = b$. \hfill\emph{(right cancellation law II)}
\end{enumerate}
\end{theorem}

\par Let $G$ be a gyrogroup. For $a\in G$, the \textit{left
gyrotranslation by $a$}, \mbox{$L_a\colon x\mapsto a\oplus x$}, and
the \textit{right gyrotranslation by $a$}, $R_a\colon x\mapsto
x\oplus a$, are permutations of $G$. Further, one has the following
composition law
\begin{equation}\label{Eqn: Gyroauto and left-right translation}
L_a\circ L_b = L_{a \oplus b}\circ\gyr{a, b}{}.
\end{equation}
From this it can be proved that every gyrogroup forms a left Bol
loop with the $\mathrm{A}_\ell$-property, where the
gyroautomorphisms correspond to \textit{left inner mappings} or
\textit{precession maps}. In fact, gyrogroups and left Bol loops
with the $\mathrm{A}_\ell$-property are equivalent, see for instance
\cite{LS1998AM}.

\section{Subgyrogroups}
\par Let $G$ be a gyrogroup. A nonempty subset $H$ of $G$
is called a \emph{subgyrogroup} if it is a gyrogroup under the
operation inherited from $G$ and the restriction of $\gyr{a, b}{}$
to $H$ becomes an automorphism of $H$ for all $a, b\in H$. If $H$ is
a subgyrogroup of $G$, we write $H\leqslant G$.  We have the
following subgyrogroup criterion, as in the group case.
\begin{proposition}[{\hskip-0.5pt}\cite{TSKW2014}]\label{Pro: Equivalent condition for subgyrogroup}
A nonempty subset $H$ of $G$ is a subgyrogroup if and only if
\emph{(1)} $a\in H$ implies $\ominus a \in H$ and \emph{(2)} $a,
b\in H$ implies $a\oplus b\in H$.
\end{proposition}

\par Subgyrogroups that arise as groups under gyrogroup operation are of great importance in the study of gyrogroups.
\begin{definition}[{\hskip-0.5pt}\cite{TFAU2000}]\label{Def: Subgroup of gyrogroup}
A nonempty subset $X$ of a gyrogroup $(G,\oplus)$ is a
\textit{subgroup} if it is a group under the restriction of $\oplus$
to $X$.
\end{definition}

\par The following proposition shows that any subgroup of a gyrogroup is
simply a subgyrogroup with trivial gyroautomorphisms.
\begin{proposition}\label{Prop: Characterization of subgroup}
A nonempty subset $X$ of a gyrogroup $G$ is a subgroup if and only
if it is a subgyrogroup of $G$ and $\res{\gyr{a, b}{}}{X} = \id{X}$
for all $a, b\in X$.
\end{proposition}

\par Just as in group theory, we obtain the following results.
\begin{proposition}\label{Prop: Intersection of subgyrogroup is a subgyrogroup}
Let $G$ be a gyrogroup and let $\cols{H}$ be a nonempty collection
of subgyrogroups of $G$. Then the intersection
$\lcap{H\in\cols{H}}{} H$ forms a subgyrogroup of $G$.
\end{proposition}
\begin{proof}
This follows directly from the subgyrogroup criterion.
\end{proof}
\begin{proposition}\label{Prop: Smallest subgyrogroup containing a set}
Let $A$ be a nonempty subset of a gyrogroup $G$. There exists a
unique subgyrogroup of $G$, denoted by $\gen{A}$, such that
\begin{enumerate}
\item $A\subseteq\gen{A}$ and
\item if $H\leqslant G$ and $A\subseteq H$, then $\gen{A}\subseteq H$.
\end{enumerate}
\end{proposition}
\begin{proof}
Set $\cols{H} = \cset{H}{H\leqslant G\textrm{ and } A\subseteq H}$.
Then $\gen{A} := \lcap{H\in\cols{H}}{}H$ is a subgyrogroup of $G$
satisfying the two conditions. The uniqueness follows from condition
(2).
\end{proof}

\par The subgyrogroup generated by one-element set $\set{a}$ is called
the \textit{cyclic subgyro-group generated by $a$}, which will be
denoted by $\gen{a}$. Next, we will give an explicit description of
$\gen{a}$.

\par Let $G$ be a gyrogroup and let $a\in G$. Define recursively the following notation:
\begin{equation}
0\cdot a = 0,\hskip0.3cm m\cdot a = a\oplus ((m-1)\cdot a),\, m \geq
1,\hskip0.3cm m\cdot a = (-m)\cdot(\ominus a),\, m < 0.
\end{equation}
We also define the right counterparts:
\begin{equation}
a\cdot 0 = 0,\hskip0.3cm a\cdot m = (a\cdot (m-1))\oplus a,\, m \geq
1,\hskip0.3cm a\cdot m = (\ominus a)\cdot(-m),\, m < 0.
\end{equation}

\begin{lemma}\label{Lem: gyr[a, a. m] = gyr[a.m, a] = gyr[a, m. a] = gyr[m.a, a] = id}
Let $G$ be a gyrogroup. For any element $a$ of $G$, $$\gyr{a\cdot m,
a}{} = \gyr{m\cdot a, a}{} = \gyr{a, m\cdot a}{} = \gyr{a, a\cdot
m}{} = \id{G}$$ for all $m\in\Z$.
\end{lemma}
\begin{proof}
By induction, $\gyr{a, a\cdot m}{} = \id{G}$ and $\gyr{a\cdot m,
a}{} = \id{G}$ for all $a\in G$ and all $m\geq 0$. By the right
gyroassociative law, $a\cdot m = m\cdot a$ for all $m\in\Z$. If $m <
0$, the left and right loop properties and the left cancellation law
together imply $\gyr{a, a\cdot m}{} = \id{G}$.
\end{proof}

\par By induction,
\begin{equation}\label{Eqn: m.a + k.a = (m+k).a, nonnegative case}
(m\cdot a) \oplus(k\cdot a) = (m+k)\cdot a
\end{equation}
for all $m, k\geq 0$. In fact, we have the following proposition.
\begin{proposition}\label{Prop: m.a + k.a = (m+k).a, integer case}
Let $a$ be an element of a gyrogroup. For all $m, k\in\Z$,
$$(m\cdot a) \oplus(k\cdot a) = (m+k)\cdot a.$$
\end{proposition}
\begin{proof}
The proof is routine, using (\ref{Eqn: m.a + k.a = (m+k).a,
nonnegative case}) and induction.
\end{proof}

\begin{theorem}\label{Thm: Explicit description of <a>}
Let $G$ be a gyrogroup and let $a\in G$. Then $\gen{a} =
\cset{m\cdot a}{m\in\Z}$. In particular, $\gen{a}$ forms a subgroup
of $G$.
\end{theorem}
\begin{proof}
Set $H = \cset{m\cdot a}{m\in\Z}$. For all $m, n\in\Z$, Proposition
\ref{Prop: m.a + k.a = (m+k).a, integer case} implies that
$\ominus(m\cdot a) = (-m)\cdot a\in H$ and $(m\cdot a)\oplus(k\cdot
a) = (m+k)\cdot a\in H$. This proves $H\leqslant G$. Since $a\in H$,
we have $\gen{a}\subseteq H$ by the minimality of $\gen{a}$. By the
closure property of subgyrogroups, $H\subseteq \gen{a}$ and so
equality holds.
\par Note that $(m\cdot a)\oplus [(n\cdot a)\oplus (k\cdot a)] =
(m+n+k)\cdot a = [(m\cdot a)\oplus(n\cdot a)]\oplus(k\cdot a)$ for
all $m, n ,k\in\Z$. Thus, $\res{\gyr{m\cdot a, n\cdot a}{}}{\gen{a}}
= \id{\gen{a}}$ for all $m, n\in\Z$ and hence $\gen{a}$ forms a
subgroup of $G$ by Proposition \ref{Prop: Characterization of
subgroup}.
\end{proof}

Theorem \ref{Thm: Explicit description of <a>} suggests us to define
the \emph{order} of an element in a gyrogroup as follows.
\begin{definition}\label{Def: order of element in a gyrogroup}
Let $G$ be a gyrogroup and let $a\in G$. The \textit{order} of $a$,
denoted by $\abs{a}$, is defined to be the cardinality of $\gen{a}$
if $\gen{a}$ is finite. In this case, we will write
$\abs{a}<\infty$. If $\gen{a}$ is infinite, the order of $a$ is
defined to be infinity, and we will write $\abs{a} = \infty$.
\end{definition}

\begin{proposition}\label{Prop: gyr[m.a, n.a] = id}
Let $G$ be a gyrogroup and let $a\in G$. For all $m, n\in\Z$,
$$\gyr{m\cdot a, n\cdot a}{} = \id{G}.$$
\end{proposition}
\begin{proof}
By induction, $L_{m\cdot a} = L_a^m$ for all $a\in G$ and all
$m\in\Z$. Since $L^{-1}_a = L_{\ominus a}$, we have from (\ref{Eqn:
Gyroauto and left-right translation}) that
$$\gyr{m\cdot a, n\cdot a}{} = L_{-(m+n)\cdot a}\circ L_{m\cdot a}\circ L_{n\cdot a}
= L_a^{-(m+n)}\circ L_a^m\circ L_a^n = \id{G}$$ for all $m, n\in\Z$.
\end{proof}
\par In light of the proof of Proposition \ref{Prop: gyr[m.a, n.a] =
id}, gyrogroups are \textit{left power alternative}. Further, the
following proposition implies that gyrogroups are \textit{power
associative}.
\begin{proposition}\label{Prop: Cyclic subgyrogroup is a cyclic group}
If $a$ is an element of a gyrogroup, then $\gen{a}$ forms a cyclic
group with generator $a$ under gyrogroup operation.
\end{proposition}
\begin{proof}
By Theorem \ref{Thm: Explicit description of <a>}, $\gen{a}$ is a
group under gyrogroup operation. By induction, $m\cdot a = a^m$ for
all $m\geq 0$, where the notation $a^m$ is used as in group theory.
If $m < 0$, one obtains similarly that $m\cdot a = a^m$. Hence,
$\gen{a}$ forms a cyclic group with generator $a$.
\end{proof}

\begin{corollary}\label{Cor: Gyrogroup generated by 1 element is cyclic}
Any gyrogroup generated by one element is a cyclic group.
\end{corollary}

\par Because the \textit{group} order of $a$ and the \textit{gyrogroup} order of $a$
are the same, we obtain the following results.
\begin{proposition}\label{Prop: Characterization of |a|}
Let $G$ be a gyrogroup and let $a\in G$.
\begin{enumerate}
\item If $\abs{a}<\infty$, then $\abs{a}$ is the smallest positive integer such that $\abs{a}\cdot a = 0$.
\item If $\abs{a} = \infty$, then $m\cdot a \ne 0$ for all $m\ne 0$ and $m\cdot a\ne k\cdot a$ for all $m\ne k$ in $\Z$.
\end{enumerate}
\end{proposition}
\begin{corollary}\label{Cor: <a>, a is of finite order}
Let $a$ be an element of a gyrogroup. If $\abs{a} = n<\infty$, then
$$\gen{a} = \set{0\cdot a, 1\cdot a,\dots, (n-1)\cdot a}.$$
\end{corollary}

\begin{corollary}\label{Cor: Order of m.a}
Let $a$ be an element of a gyrogroup and let
$m\in\Z\setminus\set{0}$.
\begin{enumerate}
\item If $\abs{a} = \infty$, then $\abs{m\cdot a} = \infty$.
\item If $\abs{a} < \infty$, then $\abs{m\cdot a} = \dfrac{\abs{a}}{\gcd{(\abs{a}, m)}}$.
\end{enumerate}
\end{corollary}

\section{Gyrogroup Homomorphisms}
\par A \textit{gyrogroup homomorphism} is a map between gyrogroups
that preserves the \mbox{gyrogroup} operations. A bijective
gyrogroup homomorphism is called a \textit{gyrogroup isomorphism}.
We say that gyrogroups $G$ and $H$ are \textit{isomorphic}, written
$G\cong H$, if there exists a gyrogroup isomorphism from $G$ to $H$.

\par Suppose that $\vphi\colon G\to H$ is a gyrogroup homomorphism.
The kernel of $\vphi$ is defined to be the inverse image of the
trivial subgyrogroup $\set{0}$ under $\vphi$. Since $\ker{\vphi}$ is
invariant under all the gyroautomorphisms of $G$, the operation
\begin{equation}\label{Eqn: Quotient gyrogroup operation}
(a\oplus\ker{\vphi})\oplus (b\oplus \ker{\vphi}) := (a\oplus
b)\oplus \ker{\vphi}, \hskip1cm a, b\in G,
\end{equation}
is independent of the choice of representatives for the left cosets.
The system $(G/\ker{\vphi}, \oplus)$ forms a gyrogroup, called a
\textit{quotient gyrogroup}. This results in the first isomorphism
theorem for gyrogroups.
\begin{theorem}[{\hskip-0.5pt}\cite{TSKW2014}, The First Isomorphism Theorem]
\label{Thm: The 1st isomorphism theorem} If $\vphi$ is a gyrogroup
homomorphism of $G$, then $G/\ker{\vphi}\cong \vphi(G)$ as
gyrogroups.
\end{theorem}

\par  A subgyrogroup $N$ of a gyrogroup $G$ is \emph{normal in $G$},
denoted by $N\unlhd G$, if it is the kernel of a gyrogroup
homomorphism of $G$. By Theorem \ref{Thm: The 1st isomorphism
theorem}, every normal subgyrogroup $N$ gives rise to the quotient
gyrogroup $G/N$, along with the \textit{canonical projection}
$\Pi\colon a\mapsto a\oplus N$.

\par We state the second isomorphism theorem for gyrogroups for easy reference; its proof can
be found in \cite{TSKW2014}.
\begin{theorem}[The Second Isomorphism Theorem]\label{Thm: The second isomorphism theorem}
Let $G$ be a gyrogroup and let $A, B\leqslant G$. If $B\unlhd G$,
then $A\oplus B\leqslant G$, $A\cap B\unlhd A$, and $(A\oplus
B)/B\cong A/(A\cap B)$ as gyrogroups.
\end{theorem}

\section{The Lagrange Property}
\par Throughout this section, all gyrogroups are finite.
A version of the Lagrange property for loops can be found in
\cite{OCMKARPV2003}. In terms of gyrogroups, the Lagrange property
can be restated as follows.
\begin{definition}\label{Def: Lagrange property}
A gyrogroup $G$ is said to have the \textit{Lagrange property} if
for each subgyrogroup $H$ of $G$, the order of $H$ divides the order
of $G$.
\end{definition}

\par A version of the following proposition for loops was proved by Bruck
in \cite{RB1971}. As the first isomorphism theorem and the second
isomorphism theorem hold for gyrogroups, we also have the following
proposition:
\begin{proposition}\label{Prop: B and K/B have LP --> K has LP}
Let $H$ be a subgyrogroup of a gyrogroup $G$ and let $B$ be a normal
subgyrogroup of $H$. If $B$ and $H/B$ have the Lagrange property,
then so has $H$.
\end{proposition}

\begin{corollary}\label{Cor: N and G/N have LP --> G has LP}
Let $N$ be a normal subgyrogroup of a gyrogroup $G$. If $N$ and
$G/N$ have the Lagrange property, then so has $G$.
\end{corollary}
\begin{proof}
Taking H = G in the proposition, the corollary follows.
\end{proof}

\begin{proposition}\label{Prop: Subgroup has the weak Lagrange property}
Let $X$ be a subgroup of a gyrogroup $G$. If $H\leqslant X$, then
$\abs{H}$ divides $\abs{X}$. In other words, every subgroup of $G$
has the Lagrange property.
\end{proposition}
\begin{proof}
Suppose that $H\leqslant X$. Since $\res{\gyr{a, b}{}}{H} = \id{H}$
for all $a, b\in H$, $H$ forms a subgroup of $G$. By definition, $X$
forms a group and $H$ becomes a subgroup of $X$. By Lagrange's
theorem for groups, $\abs{H}$ divides $\abs{X}$.
\end{proof}

\par Lagrange's theorem holds for \emph{all} gyrocommutative gyrogroups,
as shown by Baumeister and Stein in \cite[Theorem 3]{BBAS2011} in
the language of Bruck loops.
\begin{theorem}\label{Thm: Lagrange's Theorem for gyrocomm. gyrogroup}
In a gyrocommutative gyrogroup $G$, if $H\leqslant G$, then
$\abs{H}$ \mbox{divides} $\abs{G}$. In other words, every
gyrocommutative gyrogroup has the Lagrange property.
\end{theorem}
\begin{proof}
Let $G$ be a gyrocommutative gyrogroup and let $H\leqslant G$. Then
$G$ is a Bruck loop and $H$ becomes a subloop of $G$. By Theorem 3
of \cite{BBAS2011}, $\abs{H}$ divides $\abs{G}$, which completes the
proof.
\end{proof}

\par The next theorem, due to Foguel and Ungar,
enables us to extend Lagrange's theorem to all finite gyrogroups.
\begin{theorem}[{\hskip-0.5pt}\cite{TFAU2000}, Theorem 4.11]\label{Thm: Gyrogroup has a normal subgroup, Foguel and Ungar}
If $G$ is a gyrogroup, then $G$ has a normal subgroup $N$ such that
$G/N$ is a gyrocommutative gyrogroup.
\end{theorem}

\begin{theorem}[Lagrange's Theorem]\label{Thm: Lagrange's Theorem for Gyrogroup}
If $H$ is a subgyrogroup of a gyrogroup $G$, then $\abs{H}$ divides
$\abs{G}$. That is, every gyrogroup has the Lagrange property.
\end{theorem}
\begin{proof}
Let $G$ be a gyrogroup. By Theorem \ref{Thm: Gyrogroup has a normal
subgroup, Foguel and Ungar}, $G$ has a normal subgroup $N$ such that
$G/N$ is gyrocommutative. Because $N = \ker{\Pi}$, where $\Pi\colon
G\to G/N$ is the canonical projection, $N$ is a normal subgyrogroup
of $G$. By Proposition \ref{Prop: Subgroup has the weak Lagrange
property} and Theorem \ref{Thm: Lagrange's Theorem for gyrocomm.
gyrogroup}, $N$ and $G/N$ have the Lagrange property. By Corollary
\ref{Cor: N and G/N have LP --> G has LP}, $G$ has the Lagrange
property.
\end{proof}

\section{Applications}
\par In this section, we provide some applications
of Lagrange's theorem. Throughout this section, all gyrogroups are
finite.
\begin{proposition}\label{Prop: a^|a| = 1}
Let $G$ be a gyrogroup and let $a\in G$. Then $\abs{a}$ divides
$\abs{G}$. In particular, $\abs{G}\cdot a = 0$.
\end{proposition}
\begin{proof}
By definition, $\abs{a} = \abs{\gen{a}}$. By Lagrange's theorem,
$\abs{a}$ divides $\abs{G}$. Write $\abs{G} = \abs{a}k$ with
$k\in\N$, so $\abs{G}\cdot a = (\abs{a}k)\cdot a =
\underbrace{\abs{a}\cdot a\oplus \cdots \oplus \abs{a}\cdot
a}_{k\textrm{ copies}} = 0$.
\end{proof}

\par Although we know that a left Bol loop of prime order
is a cyclic group by a result of Burn \cite[Corollary 2]{RB1978}, we
present the following theorem as an application of Lagrange's
theorem.
\begin{theorem}\label{Thm: Gyrogroup of prime order}
If $G$ is a gyrogroup of prime order $p$, then $G$ is a cyclic group
of order $p$ under gyrogroup operation.
\end{theorem}
\begin{proof}
Let $a$ be a nonidentity element of $G$. Then $\abs{a}\ne 1$ and
$\abs{a}$ divides $p$. It follows that $\abs{a} = p$, which implies
$G = \gen{a}$ since $G$ is finite. By Proposition \ref{Prop: Cyclic
subgyrogroup is a cyclic group}, $\gen{a}$ is a cyclic group of
order $p$, which completes the proof.
\end{proof}

\subsection*{The Cauchy Property}
\par In the loop literature, it is known that left Bol loops of odd
order satisfy the Cauchy property \cite[Theorem 6.2]{TFMKJP2006}.
However, Bol loops fail to satisfy the Cauchy property as Nagy
proves the existence of a simple right Bol loop of exponent $2$ and
of order $96$ \cite[Corollary 3.7]{GN2009}. This also implies that
gyrogroups fail to satisfy the Cauchy property since any Bol loop of
exponent $2$ is necessarily a Bruck loop, hence a gyrocommutative
gyrogroup.

\par In this subsection, we apply Lagrange's theorem and results from loop theory to
establish that some finite gyrogroups satisfy the Cauchy property.
\begin{definition}[The Weak Cauchy Property, WCP]\label{Def: Weak Cauchy property}
\normalfont A finite gyrogroup $G$ is said to have the \textit{weak
Cauchy property} if for every prime $p$ dividing $\abs{G}$, $G$ has
an element of order $p$.
\end{definition}

\begin{definition}[The Strong Cauchy Property, SCP]\label{Def: Strong Cauchy property}
\normalfont A finite gyrogroup $G$ is said to have the
\textit{strong Cauchy property} if every subgyrogroup of $G$ has the
weak Cauchy property.
\end{definition}

\par The Cauchy property is an invariant property of
gyrogroups, as shown in the following proposition.
\begin{proposition}\label{Prop: CP is invariant under isomorphism}
Let $G$ and $H$ be gyrogroups and let $\phi\colon G\to H$ be a
gyrogroup isomorphism.
\begin{enumerate}
\item If $G$ has the weak Cauchy property, then so has $H$.
\item If $G$ has the strong Cauchy property, then so has $H$.
\end{enumerate}
\end{proposition}
\begin{proof}
(1) It suffices to prove that $\abs{\phi(a)} = \abs{a}$ for all
$a\in G$. By induction, $\phi(n\cdot a) = n\cdot \phi(a)$ for all
$a\in G$ and all $n\in\N$. Let $a\in G$. Since $\abs{a}\cdot a = 0$,
we have $\abs{a}\cdot \phi(a) = \phi(\abs{a}\cdot a) = \phi(0) = 0$.
If there were a positive integer $m < \abs{a}$ for which $m\cdot
\phi(a) = 0$, then we would have $\phi(m\cdot a) = 0$ and would have
$m\cdot a = 0$, contradicting the minimality of $\abs{a}$. Hence,
$\abs{a}$ is the smallest positive integer such that
$\abs{a}\cdot\phi(a) = 0$, which implies $\abs{\phi(a)} = \abs{a}$
by Proposition \ref{Prop: Characterization of |a|} (1).
\par (2) Let $B\leqslant H$. Set $A = \phi^{-1}(B)$. Then $A\leqslant
G$ and $A$ has the WCP. Since $\res{\phi}{A}$ is a gyrogroup
isomorphism from $A$ onto $B$, $B$ has the WCP by Item 1.
\end{proof}

\begin{corollary}\label{Cor: CP is invariant}
Let $G$ and $H$ be gyrogroups. If $G\cong H$, then $G$ has the weak
(resp. strong) Cauchy property if and only if $H$ has the weak
(resp. strong) Cauchy property.
\end{corollary}

\begin{theorem}\label{Thm: B and H/B have WCP --> H has WCP}
Let $H$ be a subgyrogroup of a gyrogroup $G$ and let $B$ be a normal
subgyrogroup of $H$.
\begin{enumerate}
\item If $B$ and $H/B$ have the weak Cauchy property, then so has $H$.
\item If $B$ and $H/B$ have the strong Cauchy property, then so has $H$.
\end{enumerate}
\end{theorem}
\begin{proof}
(1) Suppose that $p$ is a prime dividing $\abs{H}$. Since $\abs{H} =
[H\colon B]\abs{B}$, $p$ divides $\abs{H/B}$ or $\abs{B}$. If $p$
divides $\abs{B}$, then $B$ has an element of order $p$ and we are
done. We may therefore assume that $p\nmid \abs{B}$. Hence, $p$
divides $\abs{H/B}$. By assumption, $H/B$ has an element of order
$p$, say $a\oplus B$. By induction, $n\cdot (a\oplus B) = (n\cdot
a)\oplus B$ for all $n\geq 0$. Hence, by Proposition \ref{Prop:
Characterization of |a|} (1), $p$ is the smallest positive integer
such that $p\cdot a\in B$. In particular, $a\not\in B$. Note that
$\gcd{(\abs{a}, p)} = 1$ or $p$. If $\gcd{(\abs{a}, p)} = 1$ were
true, we would have $\abs{p\cdot a} = \dfrac{\abs{a}}{\gcd{(\abs{a},
p)}} = \abs{a}$, and would have $a\in\gen{a} = \gen{p\cdot
a}\leqslant B$, a contradiction. Hence, $\gcd{(\abs{a}, p)} = p$,
which implies $p$ divides $\abs{a}$. Write $\abs{a} = mp$. Then
$\abs{m\cdot a} = \dfrac{\abs{a}}{\gcd{(\abs{a}, m})} = p$, which
finishes the proof of (1).
\par (2) Suppose that $B$ and $H/B$ have the SCP. Let $A\leqslant H$. By assumption,
$A\cap B$ has the WCP. Since $A\oplus B/B \leqslant H/B$, $A\oplus
B/B$ has the WCP. Since $A/A\cap B\cong A\oplus B/B$, $A/A\cap B$
has the WCP. By Item 1, $A$ has the WCP.
\end{proof}

\begin{corollary}\label{Cor: B and G/B have WCP --> G has WCP}
Let $N$ be a normal subgyrogroup of a gyrogroup $G$. If $N$ and
$G/N$ have the weak (strong) Cauchy property, then so has $G$.
\end{corollary}

\par Consider a gyrogroup $G$ of order $pq$, where $p$ and $q$ are primes.
If $pq$ is odd, by a result of Foguel, Kinyon, and Phillips
\cite[Theorem 6.2]{TFMKJP2006}, $G$ has the weak Cauchy property.
Since any subgyrogroup of $G$ is of order $1, p, q$ or $pq$, every
subgyrogroup of $G$ has the weak Cauchy property as well. This
implies that $G$ has the strong Cauchy property. If $pq$ is even, at
least one of $p$ or $q$ must be $2$. Hence, $G$ is of order
$2\tilde{p}$, where $\tilde{p}$ is a prime. By a result of Burn
\cite[Theorem 4]{RB1978}, $G$ is a group, hence has the strong
Cauchy property. This proves the following theorem.
\begin{theorem}[Cauchy's Theorem]\label{Thm: Gyrogroup of order pq has WCP}
Let $p$ and $q$ be primes. Every gyrogroup of order $pq$ has the
strong Cauchy property.
\end{theorem}

\begin{theorem}\label{Thm: Gyrogroup of order pq has two generators}
Let $p$ and $q$ be primes and let $G$ be a gyrogroup of order $pq$.
If $p = q$, then $G$ is a group. If $p\ne q$, then $G$ is generated
by two elements; one has order $p$ and the other has order $q$.
\end{theorem}
\begin{proof}
In the case $p = q$, $G$ is a left Bol loop of order $p^2$, hence
must be a group by Burn's result \cite[Theorem 5]{RB1978}.
\par Suppose that $p\ne q$. Let $a$ and $b$ be elements of order $p$ and $q$,
\mbox{respectively}. By Lagrange's theorem, $\gen{a}\cap\gen{b} =
\set{0}$. For all $m, n, s, t\in\Z$, if $(m\cdot a)\oplus(n\cdot b)
= (s\cdot a)\oplus(t\cdot b)$, then $\ominus(s\oplus a)\oplus
(m\cdot a) = (t\cdot b)\boxplus(\ominus(n\cdot b)) = (t\cdot
b)\ominus(n\cdot b)$ belongs to $\gen{a}\cap\gen{b}$. Hence,
$\ominus(s\oplus a)\oplus (m\cdot a) = 0$ and $(t\cdot
b)\ominus(n\cdot b) = 0$ and so $m\cdot a = s\cdot a$ and $n\cdot b
= t\cdot b$. This proves $\cset{(m\cdot a)\oplus(n\cdot b)}{0\leq
m<p, 0\leq n<q}$ contains $pq$ distinct elements of $G$. Since $G$
is finite, it follows that
$$G = \cset{(m\cdot a)\oplus(n\cdot b)}{0\leq m<p, 0\leq n<q} = \gen{a, b}.\eqno\qedhere$$
\end{proof}

\par In general, gyrogroups of order $pq$, where $p$ and $q$ are distinct primes
not equal to $2$, need not be groups. This is a situation where
gyrogroups are different from Moufang loops. As Moufang loops are
\textit{diassociative}, every Moufang loop generated by two elements
must be a group. This implies that Moufang loops of order $pq$ are
groups \cite[Proposition 3]{OC1974}.

\par Let $G$ be a finite \textit{nongyrocommutative} gyrogroup. By
Theorem  \ref{Thm: Gyrogroup has a normal subgroup, Foguel and
Ungar}, $G$ has a normal subgroup $N$ such that $G/N$ is
gyrocommutative. Because $G$ is nongyrocommutative, we have $N$ is
nontrivial, since otherwise $\Pi\colon G\to G/N$ would be a
gyrogroup isomorphism and $G$ and $G/N$ would be isomorphic
gyrogroups. From this we can deduce the following results.

\begin{theorem}\label{Thm: CP for Nongyrocomm. gyrogroup of order p^3} Let $p$ be a prime. Every
nongyrocommutative gyrogroup of order $p^3$ has the strong Cauchy
property.
\end{theorem}
\begin{proof}
Let $G$ be a nongyrocommutative gyrogroup of order $p^3$. As noted
above, $G$ has a nontrivial normal subgroup $N$. By Lagrange's
theorem, $\abs{N} = p, p^2$ or $p^3$. If $\abs{N} = p^3$, then $G =
N$ is a group, hence has the SCP. If $\abs{N}\in\set{p, p^2}$, then
$\abs{N}\in\set{p, p^2}$. In any case, $N$ and $G/N$ form groups.
Hence, $N$ and $G/N$ have the SCP and by Corollary \ref{Cor: B and
G/B have WCP --> G has WCP}, $G$ has the SCP.
\end{proof}

\begin{theorem}\label{Thm: CP for Nongyrocomm. gyrogroup of order pqr}
Let $p, q$ and $r$ be primes. Every nongyrocommutative gyrogroup of
order $pqr$ has the strong Cauchy property.
\end{theorem}
\begin{proof}
The proof follows the same steps as in the proof of Theorem
\ref{Thm: CP for Nongyrocomm. gyrogroup of order p^3}.
\end{proof}

\section*{Acknowledgements}
This work was completed with the support of Development and
Promotion of \mbox{Science} and Technology Talents Project (DPST),
Institute for Promotion of \mbox{Teaching} Science and Technology
(IPST), Thailand.

\bibliographystyle{amsplain}
\bibliography{Lagrange_Theorem}
\end{document}